\documentclass[12pt]{amsart}

\usepackage{amsmath}
\usepackage{amstext}
\usepackage{amsfonts}
\usepackage{amssymb, dsfont}
\usepackage{amsthm}
\usepackage{amsrefs}
\usepackage{mathtools}
\usepackage{color}
\usepackage{graphicx}
\usepackage[colorlinks=true, allcolors=blue]{hyperref}
\usepackage{float}
\usepackage{MnSymbol}
\allowdisplaybreaks
\usepackage{tikz-cd}
\usepackage{amsbsy}
\usepackage{wasysym}
\usepackage{amscd}
\usetikzlibrary{matrix,arrows,decorations.pathmorphing}
\usepackage{enumitem}
\usepackage{setspace}

\usepackage[english]{babel}
\usepackage[utf8x]{inputenc}
\usepackage[T1]{fontenc}

\usepackage[margin=1.0in]{geometry}

\newtheorem{thm}{Theorem}[section]
\newtheorem{lem}[thm]{Lemma}
\newtheorem{cor}[thm]{Corollary}
\newtheorem{prop}[thm]{Proposition}

\theoremstyle{definition}
\newtheorem{example}[thm]{Example}
\theoremstyle{definition}

\theoremstyle{definition}

\theoremstyle{definition}
\newtheorem{remark}[thm]{Remark}

\newcommand{\LG}{VN(G)}

\newcommand{\LO}{L^1(G)}
\newcommand{\LOQ}{L^1(\mathbb{G})}

\newcommand{\LOQH}{L^1(\widehat{\mathbb{G}})}

\newcommand{\LT}{L^2(G)}

\newcommand{\LTQ}{L^2(\mathbb{G})}
\newcommand{\LTQH}{L^2(\widehat{\mathbb{G}})}

\newcommand{\LI}{L^{\infty}(G)}
\newcommand{\LIQ}{L^{\infty}(\mathbb{G})}

\newcommand{\LIQH}{L^{\infty}(\widehat{\mathbb{G}})}

\newcommand{\LIQHP}{L^{\infty}(\widehat{\mathbb{G}}')}

\newcommand{\BLT}{\mc{B}(L^2(G))}

\newcommand{\BLTQ}{\mc{B}(L^2(\mathbb{G}))}

\newcommand{\TC}{\mc{T}(L^2(G))}

\newcommand{\TCQ}{\mc{T}(L^2(\mathbb{G}))}

\newcommand{\McbQHr}{M_{cb}^r(L^1(\widehat{\mathbb{G}}))}

\newcommand{\McbQr}{M_{cb}^r(L^1(\mathbb{G}))}

\newcommand{\ten}{\otimes}
\newcommand{\oten}{\overline{\otimes}}
\newcommand{\pten}{\widehat{\otimes}}

\newcommand{\Nphi}{\mc{N}_\varphi}
\newcommand{\Mphi}{\mc{M}_\varphi}
\newcommand{\Mphip}{\mc{M}_{\varphi}^+}

\newcommand{\vphi}{\varphi}

\newcommand{\G}{\mathbb{G}}

\newcommand{\C}{\mathbb{C}}

\newcommand{\R}{\mathbb{R}}

\DeclareSymbolFont{lettersA}{U}{txmia}{m}{it}
\DeclareMathSymbol{\W}{\mathord}{lettersA}{151}

\newcommand{\Lphi}{\Lambda_\varphi}
\newcommand{\Lphit}{\Lambda_{\varphi\otimes\varphi}}

\newcommand{\lm}{\lambda}

\newcommand{\Gam}{\Gamma}
\newcommand{\om}{\omega}

\newcommand{\Ad}{\mathrm{Ad}}

\newcommand{\id}{\textnormal{id}}

\providecommand{\norm}[1]{\lVert#1\rVert}

\newcommand{\h}[1]{\widehat{#1}}
\newcommand{\wh}[1]{\widehat{#1}}

\newcommand{\mc}[1]{\mathcal{#1}}
\newcommand{\e}[1]{\emph{#1}}

\newcommand{\la}{\langle}
\newcommand{\ra}{\rangle}
\newcommand{\tr}{\mathrm{tr}}
\newcommand{\rmv}[1]{}

\newcommand{\hs}{\hskip10pt}

\begin{document}

\title[]{A Banach algebra encoding quantum group duality}
\author{Jason Crann$^1$}
\email{jasoncrann@cunet.carleton.ca}

\author{Matthias Neufang$^{1,2}$}
\email{matthias.neufang@carleton.ca}
\address{$^1$School of Mathematics and Statistics, Carleton University, Ottawa, ON, Canada H1S 5B6}
\address{$^2$Universit\'{e} Lille 1 - Sciences et Technologies, UFR de Math\'{e}matiques, Laboratoire de Math\'{e}\-matiques Paul Painlev\'{e} - UMR CNRS 8524, 59655 Villeneuve d'Ascq C\'{e}dex, France}


\begin{abstract} We introduce and study a new Banach algebra structure on the trace-zero subspace $\TCQ_0$ of trace class operators for any locally compact quantum group $\G$; it is defined through a mixed Lie-type product $\ostar$ of the two dual products on $\TCQ$ arising from the canonical extensions of the co-products of $\G$ and $\wh{\G}$. The surprising fact that this new product is indeed associative stems precisely from the duality of the latter two products. This, in particular, gives new faithful associative products on trace-zero matrices in $M_d(\C)$. After establishing some basic properties of $\ostar$, we show that the single algebra $(\TCQ_0,\ostar)$ captures simultaneous properties of $\G$ and $\wh{\G}$, is faithful for a large class of quantum groups, and encodes both $\McbQr$ and $\McbQHr$ as left, respectively right, completely bounded module maps on $\TCQ$. We finish by exhibiting an analogous product on the trace-zero nuclear operators $\mc{N}(L^p(G))_0$ for a locally compact group $G$ and $p\in(1,\infty)$. Building on \cite{DS}, our work suggests an approach for developing an $L^p$-version of locally compact quantum group theory.
\end{abstract}

\begin{spacing}{1.0}

\maketitle

\section{Introduction}

For any locally compact quantum group $\G$, the preduals of the canonical extensions of the (right) co-products on $\LIQ$ and $\LIQH$ yield two Banach algebra structures on the trace class $\TCQ$. These products, denoted here by $\star$ and $\bullet$, lift the convolution products of $\LOQ$ and $\LOQH$, respectively. They have been studied from a variety of perspectives, including Banach algebra homology \cite{C,C2,CN,KN,P}, completely bounded multipliers \cite{HNR2,HNR3,N}, quantum group regularity \cite{HNR4} and topological centres \cite{HNR, HNR5}. 

While $\star$ and $\bullet$ capture fundamental properties of $\G$ and $\wh{\G}$, respectively, the products themselves are degenerate in the sense that $\rho\star\tau=0$ for any $\tau\in\LIQ_{\perp}$ and $\rho\in\TCQ$ (similarly for $\bullet$). In this paper we show, surprisingly, that the following mixed Lie-type product
$$\rho\ostar\tau := \frac{1}{2}(\rho\star\tau - \tau\bullet\rho)$$
yields an associative (!), completely contractive Banach algebra structure on the subspace $\TCQ_0$ of trace-zero elements of $\TCQ$. We show that for a large class of locally compact quantum groups $\G$, the product $\ostar$ is (left and right) faithful and that simultaneous properties of both $\G$ and $\wh{\G}$ are captured by the single algebra $(\TCQ_0,\ostar)$. For instance, we show in Theorem \ref{t:bai} that $(\TCQ_0,\ostar)$ has a bounded approximate identity if and only if both $\G$ and $\wh{\G}$ are co-amenable.

The compatibility of the products $\star$ and $\bullet$ that ensures the associativity of $\ostar$ is mutual commutativity on $\TCQ_0$ (see Proposition \ref{p:general} for a precise meaning). We show that instances of this mutual commutativity occur beyond the Hilbert space setting by establishing an analogous product to $\ostar$ on the trace-zero nuclear operators $\mc{N}(L^p(G))_0$ for a locally compact group $G$ when $p\in (1,\infty)$ (see section \ref{s:last}). This result utilizes and builds on techniques developed by Neufang \cite{N} and Daws-Spronk \cite{DS}.  

For any locally compact quantum group $\G$, it is well-established that both $\McbQr$ and $\McbQHr$ are represented completely isometrically in $\mc{CB}(\TCQ)$ as completely bounded left $\star$, respectively $\bullet$ module maps. Specifically,
$$\McbQr\cong \ _{\star}\mc{CB}(\TCQ) \ \ \textnormal{and} \ \ \McbQHr\cong \ _{\bullet}\mc{CB}(\TCQ)$$
anti-multiplicatively \cite[Theorem 4.1]{HNR4}. We show that $\TCQ$ becomes a $(\TCQ_0,\ostar)$-bimodule in the canonical manner, and that for a large class of quantum groups $\G$, the image of the above representations induce completely isometric isomorphisms
$$\McbQr\cong \ _{\ostar}\mc{CB}(\TCQ)\ \ \textnormal{and}  \ \ \McbQHr\cong \ \mc{CB}(\TCQ)_{\ostar}.$$
In other words, the completely bounded multipliers of $\LOQ$ and $\LOQH$ can be viewed as the left, respectively right, completely bounded module maps of $\TCQ$ under the \textit{same product} (see Theorem \ref{t:multipliers} for a precise statement).

The structure of the paper is as follows. After reviewing necessary preliminaries in section 2, we introduce the Banach algebra $(\TCQ_0,\ostar)$ and study some of its basic properties in section 3. We characterize the existence of a bounded approximate identity in section 4, and establish the aforementioned representations of completely bounded multipliers in section 5. In section 6, we study analogous products to $\ostar$ on the trace zero nuclear opeators $\mc{N}(L^p(G))_0$ for a locally compact group $G$ and $p\in(1,\infty)$.

\section{Preliminaries}

\subsection{Completely Contractive Banach Algebras}

A completely contractive Banach algebra $A$ is a Banach algebra for which multiplication extends to a complete contraction $m_A:A\pten A\rightarrow A$ with respect to the operator space projective tensor product $\pten$. Equivalently,
$$\norm{[a_{i,j}b_{k,l}]}_{mn}\leq\norm{[a_{i,j}]}_{m}\norm{[b_{k,l}]}_n, \ \ \ [a_{i,j}]\in M_m(A), \ [b_{k,l}]\in M_n(A).$$
An operator space $X$ is a left \textit{operator $A $-module} if it is a left Banach $ A  $-module such that the module map $m_X:A\pten X\rightarrow X$ is completely contractive. We say that $X$ is \textit{faithful} if for every non-zero $x\in X$, there is $a\in A  $ such that $a\cdot x\neq 0$, and we say that $X$ is \e{essential} if $\la A\cdot X \ra=X$, where $\la\cdot\ra$ denotes the closed linear span.

The operator space dual $X^*$ of any left operator $A$-module becomes a right operator $A$-module in the canonical fashion:
\begin{equation}\label{e:dualaction}\la f\cdot a, x\ra=\la f, a\cdot x\ra, \ \ \ f\in X^*, \ a\in A, \ x\in X.\end{equation}
Similarly for right modules.

\subsection{Locally Compact Quantum Groups} A \e{locally compact quantum group} is a quadruple $\G=(\LIQ,\Gam,\vphi,\psi)$, where $\LIQ$ is a Hopf-von Neumann algebra with co-multiplication $\Gam:\LIQ\rightarrow\LIQ\oten\LIQ$, and $\vphi$ and $\psi$ are fixed left and right Haar weights on $\LIQ$, respectively \cite{KV1,KV2}. Throughout the paper we adopt the following standard notations from weight theory
$$\Mphip:=\{x\in \LIQ^+ : \vphi(x)<\infty\}, \ \ \ \Nphi:=\{x\in\LIQ : \vphi(x^*x)<\infty\}.$$ 
It is well-known that $\Nphi$ is a left ideal in $\LIQ$, and that $\Mphi:=\text{span}\{x^*y : x,y\in\Nphi\}=\text{span} \ \Mphip$ is a *-subalgebra of $\LIQ$ \cite{T2}. Similarly for the right Haar weight $\psi$.

For every locally compact quantum group $\G$, there exists a \e{left fundamental unitary operator} $W$ on $L^2(\G,\vphi)\ten L^2(\G,\vphi)$ and a \e{right fundamental unitary operator} $V$ on $L^2(\G,\psi)\ten L^2(\G,\psi)$ implementing the co-multiplication $\Gam$ via
\begin{equation*}\Gam(x)=W^*(1\ten x)W=V(x\ten 1)V^*, \ \ \ x\in\LIQ.\end{equation*}
Both unitaries satisfy the \e{pentagonal relation}; that is,
\begin{equation}\label{e:pentagonal}W_{12}W_{13}W_{23}=W_{23}W_{12}\hs\hs\mathrm{and}\hs\hs V_{12}V_{13}V_{23}=V_{23}V_{12}.\end{equation}

For simplicity we write $\LTQ$ for $L_2(\G,\vphi)$ throughout the paper. There is a strictly positive operator $\delta$ affiliated with $\LIQ$, called the \textit{modular element}, satisfying $\psi(x)=\vphi(\delta^{1/2}x\delta^{1/2})$, $x\in \mc{M}_\psi$. 

Let $\LOQ$ denote the predual of $\LIQ$. Then the pre-adjoint of $\Gam$ induces an associative completely contractive multiplication on $\LOQ$, defined by
\begin{equation*}\star:\LOQ\pten\LOQ\ni f\ten g\mapsto f\star g=\Gam_*(f\ten g)\in\LOQ.\end{equation*}
The multiplication $\star$ is a complete quotient map from $\LOQ\pten\LOQ$ onto $\LOQ$, implying
\begin{equation*}\la\LOQ\star\LOQ\ra=\LOQ,\end{equation*}
where as above, $\la\cdot\ra$ denotes closed linear span. 


The canonical $\LOQ$-bimodule structure on $\LIQ$ is given by
\begin{equation*}f\star x=(\id\ten f)\Gam(x)\hs\hs\mathrm{and}\hs\hs x\star f=(f\ten\id)\Gam(x)\end{equation*}
for $x\in\LIQ$, and $f\in\LOQ$. A \e{left invariant mean on $\LIQ$} is a state $m\in \LIQ^*$ satisfying
\begin{equation*}\la m,x\star f \ra=\la f,1\ra\la m,x\ra, \ \ \ x\in\LIQ, \ f\in\LOQ.\end{equation*}
Right and two-sided invariant means are defined similarly. A locally compact quantum group $\G$ is said to be \e{amenable} if there exists a left invariant mean on $\LIQ$. It is known that $\G$ is amenable if and only if there exists a right (equivalently, two-sided) invariant mean. We say that $\G$ is \e{co-amenable} if $\LOQ$ has a bounded left (equivalently, right or two-sided) approximate identity. We say that $\G$ is \e{discrete} if $\LOQ$ is unital, in which case we denote $\LOQ$ by $\ell^1(\G)$.

For general $\G$, the \e{left regular representation} $\lm:\LOQ\rightarrow\BLTQ$ is defined by
\begin{equation*}\lm(f)=(f\ten\id)(W), \ \ \ f\in\LOQ,\end{equation*}
and is an injective, completely contractive homomorphism from $\LOQ$ into $\BLTQ$. Then $\LIQH:=\{\lm(f) : f\in\LOQ\}''$ is the von Neumann algebra associated with the dual quantum group $\h{\G}$. Analogously, we have the \e{right regular representation} $\rho:\LOQ\rightarrow\BLTQ$ defined by
\begin{equation*}\rho(f)=(\id\ten f)(V), \ \ \ f\in\LOQ,\end{equation*}
which is also an injective, completely contractive homomorphism from $\LOQ$ into $\BLTQ$. Then $\LIQHP:=\{\rho(f) : f\in\LOQ\}''$ is the von Neumann algebra associated to the quantum group $\h{\G}'$. It follows that $\LIQHP=\LIQH'$ and $\LIQ\cap\LIQH=\LIQ\cap\LIQHP=\C1$ \cite[Proposition 3.4]{VD}. The left and right fundamental unitaries of the dual $\h{\G}$ are given by
\begin{equation}\label{e:hat}\h{W}=\sigma W^*\sigma, \ \ \ \h{V}=(J\ten J)W(J\ten J),\end{equation}
where $\sigma:\LTQ\ten\LTQ\rightarrow\LTQ\ten\LTQ$ is the flip map and $J$ is the modular conjugation induced by the left Haar weight $\vphi$. The GNS Hilbert space of the dual left Haar weight $\h{\vphi}$ on $\LIQH$ can be identified with $\LTQ$. In this case, the product of the modular conjugations arising from the left Haar weights $\vphi$ and $\h{\vphi}$ defines a unitary $U:=\h{J}J$ on $\LTQ$ intertwining the left and right regular representations via $\rho(f)=U\lm(f)U^*$, $f\in\LOQ$. At the level of fundamental unitaries, this relation becomes
\begin{equation}\label{e:WV}V=\sigma(1\ten U)W(1\ten U^*)\sigma.\end{equation}
Equations (\ref{e:WV}) and (\ref{e:hat}) allow us to interpret the pentagonal relation as a commutation relation between $\G$ and $\h{\G}$:
\begin{equation}\label{e:commutation}\h{W}_{23}W_{13}=W_{13}\h{W}_{23}W_{12}, \ \ \ \h{V}_{12}V_{13}=V_{13}\h{V}_{12}\h{W}_{23}^*.\end{equation}

If $G$ is a locally compact group, we let $\G_a=(\LI,\Gam_a,\vphi_a,\psi_a)$ denote the \e{commutative} quantum group associated with the commutative von Neumann algebra $\LI$, where the co-multiplication is given by $\Gam_a(f)(s,t)=f(st)$, and $\vphi_a$ and $\psi_a$ are integration with respect to a left and right Haar measure, respectively. The fundamental unitaries in this case are given by
\begin{equation}\label{e:W_a}W_a\xi(s,t)=\xi(s,s^{-1}t), \ \ V_a\xi(s,t)=\xi(st,t)\Delta(t)^{1/2}, \ \ \ \xi\in L^2(G\times G).\end{equation}
The dual $\h{\G}_a$ of $\G_a$ is the \e{co-commutative} quantum group $\G_s=(\LG,\Gam_s,\vphi_s,\psi_s)$, where $\LG$ is the left group von Neumann algebra with co-multiplication $\Gam_s(\lm(t))=\lm(t)\ten\lm(t)$, and $\vphi_s=\psi_s$ is Haagerup's Plancherel weight. Then $L^1(\G_a)$ is the usual group convolution algebra $\LO$, and $L^1(\G_s)$ is the Fourier algebra $A(G)$. It is known that every commutative locally compact quantum group is of the form $\G_a$ \cite[Theorem 2]{T}. By duality, every co-commutative locally compact quantum group is of the form $\G_s$.

For general $\G$, we let $C_0(\G):=\overline{\hat{\lm}(\LOQH)}^{\norm{\cdot}}$ denote the \e{reduced quantum group $C^*$-algebra} of $\G$. Here,
$$\hat{\lm}(\hat{f})=(\hat{f}\ten\id)(\wh{W}), \ \ \ \hat{f}\in\LOQH,$$
where $\wh{W}=\sigma W^*\sigma$ is the left fundamental unitary of $\wh{\G}$, and $\sigma$ is the flip map on $\LTQ\ten\LTQ$. The operator dual $M(\G):=C_0(\G)^*$ is a completely contractive Banach algebra containing $\LOQ$ as a norm closed two-sided ideal via the map $\LOQ\ni f\mapsto f|_{C_0(\G)}\in M(\G)$.

For a locally compact quantum group $\G$, we let $R$ and $(\tau_t)_{t\in\R}$ denote the \textit{unitary antipode} and \textit{scaling group} of $\G$, respectively. The unitary antipode satisfies
\begin{equation}\label{e:antipode} (R\ten R)\circ\Gam = \Sigma\circ\Gamma\circ R,\end{equation}
where $\Sigma:\LIQ\oten\LIQ\rightarrow\LIQ\oten\LIQ$ denotes the flip map. The \textit{antipode} of $\G$ is $S=R\tau_{-\frac{i}{2}}$, and is a closed densely defined operator on $\LIQ$, whose domain we denote by $\mc{D}(S)$. 
A quantum group $\G$ is said to be a \textit{Kac algebra} if $S=R$, and the modular element $\delta$ is affiliated to the center of $\LIQ$. In this case $\h{\G}$ is also a Kac algebra.

\section{Dual Products on the Trace Class}

Let $\G=(\LIQ,\Gamma,\vphi,\psi)$ be a locally compact quantum group. The right fundamental unitary $V$ of $\G$ induces a co-associative co-multiplication
\begin{equation*}\Gam^r:\BLTQ\ni T\mapsto V(T\ten 1)V^*\in\BLTQ\oten\BLTQ,\end{equation*}
and the restriction of $\Gam^r$ to $\LIQ$ yields the original co-multiplication $\Gam$ on $\LIQ$. The pre-adjoint of $\Gam^r$ induces an associative completely contractive multiplication on $\TCQ$, defined by
\begin{equation*}\star:\TCQ\pten\TCQ\ni\om\ten\tau\mapsto\om\star\tau=\Gam^r_*(\om\ten\tau)\in\TCQ.\end{equation*}
Since $\Gam^r$ is a complete isometry, it follows that $\Gam_*^r$ is a complete quotient map, so we have
\begin{equation}\label{e:densityofproducts}\TCQ=\la\TCQ\star\TCQ\ra.\end{equation}
It is known that $(\TCQ,\star)$ is always left faithful, and right faithful if and only if $\G$ is trivial (cf. \cite{HNR2}). By \cite[Lemma 5.2]{HNR2}, the pre-annihilator $\LIQ_{\perp}$ of $\LIQ$ in $\TCQ$ is a norm closed two sided ideal in $(\TCQ,\star)$ and the complete quotient map
\begin{equation}\label{pi}\pi:\TCQ\ni\om\mapsto f=\om|_{\LIQ}\in\LOQ\end{equation}
is a completely contractive algebra homomorphism from $(\TCQ,\star)$ onto $\LOQ$. Therefore, we have the completely isometric Banach algebra identification
\begin{equation*}(\LOQ,\star)\cong(\TCQ,\star)/\LIQ_{\perp}.\end{equation*}
This allows us to view $(\TCQ,\star)$ as a lifting of $(\LOQ,\star)$.


\begin{example}\label{ex:1} Let $\G_a=(\LI,\Gam_a,\vphi_a,\psi_a)$ be a commutative locally compact quantum group. Then, as shown in \cite{N}, we have
$$\om\star_a\tau=\int_G\rho(s)^*\om\rho(s)\pi_a(\tau)(s)$$
for any $\om,\tau\in\TC$. Hence, in the commutative case we may view $\star_a$ as a lifting of convolution from $\LO$ to $\TC$.
\end{example}

\begin{example}\label{ex:2} Let $\G_s=(\LG, \Gam_s,\vphi_s)$ be a co-commutative quantum group. By equations (\ref{e:hat}) and (\ref{e:W_a}), it follows that the right fundamental unitary $V_s=W_a$. Then, for any $\rho,\tau\in\TCQ$ and $T\in\BLTQ$
\begin{align*}\la\rho\star_s\tau,T\ra&=\la\rho\ten\tau,V_s(T\ten 1)V_s^*\ra\\
&=\la\rho\ten\tau,W_a(T\ten 1)W_a^*\ra\\
&=\la W_a^*(\rho\ten\tau)W_a,T\ten 1\ra\\
&=\la (\id\ten\tr)W_a^*(\rho\ten\tau)W_a,T\ra.
\end{align*}
Suppose that $G$ is discrete. Then using the canonical orthonormal basis $(\delta_s)_{s\in G}$ of $\ell^2(G)$, we see that
\begin{align*}\la(\rho\star_s\tau)\delta_t,\delta_s\ra&=\sum_{r\in G}\la W_a^*(\rho\ten\tau)W_a\delta_t\ten\delta_r,\delta_s\ten\delta_r\ra\\
&=\sum_{r\in G}\la (\rho\ten\tau)W_a\delta_t\ten\delta_r,W_a\delta_s\ten\delta_r\ra\\
&=\sum_{r\in G}\la (\rho\ten\tau)\delta_t\ten\delta_{tr},\delta_s\ten\delta_{sr}\ra\\
&=\la\rho\delta_t,\delta_s\ra \sum_{r\in G}\la\tau\delta_{tr},\delta_{sr}\ra\\
&=\la\rho\delta_t,\delta_s\ra \sum_{r\in G}\la\lm(s^{-1})\tau\lm(t)\delta_{r},\delta_{r}\ra\\
&=\la\rho\delta_t,\delta_s\ra \tr(\lm(s^{-1})\tau\lm(t))\\
&=\la\rho\delta_t,\delta_s\ra \tr(\tau\lm(ts^{-1}))\\
&=\la([\tr(\tau\lm(ts^{-1}))]_{s,t}\circ_S\rho)\delta_t,\delta_s\ra,
\end{align*}
where $\circ_S$ is the (entry-wise) Schur product.
\end{example}

Since $\LTQ=\LTQH$ for any locally compact quantum group $\G$, applying the above construction to the co-multiplication $\h{\Gam}$ on $\LIQH$ yields a \e{dual} product
$$\bullet:\TCQ\pten\TCQ\ni\om\ten\tau\mapsto\om\bullet\tau=\h{\Gam}^r_*(\om\ten\tau)\in\TCQ.$$

Lifting convolution from both $\LOQ$ and $\LOQH$ to $\TCQ$ allows one to study properties of $\G$ and $\widehat{\G}$ as well as their interactions a \e{single space}. Two important interactions were shown in \cite{KN}, which will be used to establish associativity of our Lie-type product.

\begin{prop}\cite{KN} Let $\G$ be a locally compact quantum group. Then for every $\rho,\om,\tau\in\TCQ$ we have
\begin{equation}\label{e:dr}\rho \star (\om\bullet \tau)=\tau(1)\rho \star \om, \quad \rho\bullet(\om\star\tau)=\tau(1)\rho\bullet\om,\end{equation}
and
\begin{equation}\label{e:cr}(\rho \star \om)\bullet \tau=(\rho \bullet \tau) \star \om.\end{equation}
\end{prop}

Note that the trace $\tr:\TCQ\rightarrow\C$ is multiplicative with respect to both products.

We now derive a relation between dual fundamental unitaries which leads to a concrete formula relating the dual trace class products. Let $m_{\star}:\TCQ\pten\TCQ\to\TCQ$ denote multiplication and similarly for $m_{\bullet}$. 

\begin{prop}\label{p:products} Let $\G$ be a locally compact quantum group. There exists a unitary $u\in\BLTQ$ for which $\wh{V}=\sigma WV(1\ten u)$. Consequently, 
\begin{equation}\label{e:products}m_{\bullet}=m_{\star}\circ\Ad(W^*)\circ\Sigma.\end{equation}
When $\G$ is a unimodular Kac algebra for which $\vphi$ is tracial, we have $u=U=\wh{J}J$. 
\end{prop}

\begin{proof} Fix $x\in\LIQ$ and $\h{x}\in\LIQH$. Then
\begin{align*}\wh{V}(x\h{x}\ten 1)\wh{V}^*&=(x\ten 1)\wh{\Gamma}(\hat{x})\\
&=\sigma W\Gamma(x)W^*\sigma\wh{\Gamma}(\hat{x})\\
&=\sigma W\Gamma(x)W^*\sigma(\wh{W}^*(1\ten \h{x})\wh{W})\\
&=\sigma W\Gamma(x)(\h{x}\ten 1)\sigma\wh{W}\\
&=\sigma WV(x\h{x}\ten 1)V^*W^*\sigma.
\end{align*}
By weak* density of $\la\LIQ\LIQH\ra$ in $\BLTQ$ \cite[Proposition 2.5]{VV}, it follows that 
$$\wh{V}(T\ten 1)\wh{V}^*=\sigma WV(T\ten 1)V^*W^*\sigma, \ \ \ T\in\BLTQ,$$
which implies the existence of a (necessarily unique unitary) element $u\in\BLTQ$ such that $\wh{V}=\sigma WV(1\ten u)$. At the level of co-multiplications, this implies
$$\wh{\Gamma}^r(T)=\sigma WV(1\ten u)(T\ten 1)(1\ten u)^*V^*W^*\sigma=\sigma W\Gamma^r(T)W^*\sigma, \ \ \ T\in\BLTQ,$$
that is, $\wh{\Gamma}^r=\Sigma\circ\Ad(W)\circ\Gamma^r$. Taking pre-adjoints, it follows that $m_{\bullet}=m_{\star}\circ\Ad(W^*)\circ\Sigma$. 

Finally, suppose that $\G$ is a unimodular Kac algebra for which $\vphi$ is tracial. We will show that $u$ is necessarily $U=\wh{J}J$ by showing directly that $\wh{V}=\sigma WV(1\ten U)$.

Let $a,b,c,d\in\Nphi$. Using the relation $\wh{V}=(J\ten J)W(J\ten J)$, we have
\begin{align*}
&\la V(1\ten U)\wh{V}^*\sigma\Lphi(a)\ten \Lphi(b),V\Lphi(c)\ten \Lphi(d)\ra\\
&= \la (1\ten U)\wh{V}^*\sigma\Lphi(a)\ten \Lphi(b),\Lphi(c)\ten \Lphi(d)\ra\\
&=\la (J\ten \wh{J})W^*(\Lphi(b^*)\ten\Lphi(a^*)),\Lphi(c)\ten\Lphi(d)\ra\\
&=\la \Lphi(c^*)\ten\Lphi(R(d^*)),W^*(\Lphi(b^*)\ten\Lphi(a^*))\ra\\
&=\la \Lphi(c^*)\ten\Lphi(R(d^*)),\Lphit(\Gamma(a^*)(b^*\ten 1))\ra\\
&=(\vphi\ten\vphi)((b\ten 1)\Gamma(a)(c^*\ten R(d^*)))\\
&=R(d^*)\cdot\vphi((\vphi\ten\id)(\Gamma(a)(c^*b\ten 1)))\\
&=R(d^*)\cdot\vphi(R((\vphi\ten\id)((a\ten 1)\Gamma(c^*b)))) \ \ \ \textnormal{(\cite[Proposition 5.24]{KV1})}\\
&=(\vphi\ten\vphi)((a\ten d^*)\Gamma(c^*b))\\
&=(\vphi\ten\vphi)((1\ten d^*)\Gamma(c^*b)(a\ten 1))\\
&=\la\Lphit(\Gamma(b)(a\ten 1)),\Lphit(\Gamma(c)(1\ten d))\ra\\
&=\la W^*\Lphi(a)\ten\Lphi(b),V\Lphi(c)\ten\Lphi(d)\ra.
\end{align*}
It follows that $W^*=V(1\ten U)\wh{V}^*\sigma$, equivalently, $\wh{V}=\sigma WV(1\ten U)$. 

\end{proof}

\begin{remark} The identity $\wh{V}=\sigma WV(1\ten U)$ could potentially hold for more general quantum groups $\G$.
\end{remark}

\section{Lie/Jordan Products Encoding Quantum Group Duality}

Throughout this section, unless otherwise stated, $\G$ denotes a locally compact quantum group. We now define Joran/Lie-type products on $\TCQ$ using both $\star$ and its dual $\bullet$. In what follows, we let $\TCQ_0:=\mathrm{Ker}(\tr)$. 

\begin{thm}\label{t:Jordan} The maps $\ostar,\ostar^+:\TCQ\times\TCQ\to\TCQ$ given by
\begin{align}\label{e:product}\rho\ostar\tau&:=\frac{1}{2}(\rho\star\tau-\tau\bullet\rho), \ \ \ \rho,\tau\in\TCQ,\\
\label{e:jordan}\rho\ostar^+\tau&:=\frac{1}{2}(\rho\star\tau+\tau\bullet\rho), \ \ \ \rho,\tau\in\TCQ,\end{align}
are bilinear, and define associative completely contractive Banach algebra products on $\TCQ_0$.
\end{thm}

\begin{proof} We first verify bilinearity of $\ostar$: given $\rho,\om,\tau\in\TCQ$ and $a,b\in\C$, we have
\begin{align*}\rho\ostar(a\om+b\tau)&=\frac{1}{2}(\rho\star(a\om+b\tau)-(a\om+b\tau)\bullet\rho)\\
&=\frac{1}{2}(a\rho\star\om+b\rho\star\tau-a\om\bullet\rho-b\tau\bullet\rho)\\
&=a\rho\ostar\om+b\rho\ostar\tau.
\end{align*}
Linearity in the first argument is identical.

Now, let $\rho,\om,\tau\in\TCQ_0$. We calculate:
\begin{align*}
(\rho\ostar\om)\ostar\tau&=\frac{1}{2}((\rho\ostar\om)\star\tau-\tau\bullet(\rho\ostar\om))\\
&=\frac{1}{4}((\rho\star\om-\om\bullet\rho)\star\tau-\tau\bullet(\rho\star\om-\om\bullet\rho)).
\end{align*}
By (\ref{e:dr}), $\tau\bullet(\rho\star\om)=\om(1)\tau\bullet\rho=0$, and by (\ref{e:cr}) $(\om\bullet\rho)\star\tau=(\om\star\tau)\bullet\rho$. Together with associativity of the individual products, we see that
$$(\rho\ostar\om)\ostar\tau=\frac{1}{4}(\rho\star(\om\star\tau)-(\om\star\tau)\bullet\rho+(\tau\bullet\om)\bullet\rho).$$
By adding $0=-\om(1)\rho\star\tau=-\rho\star(\tau\bullet\om)$ (using (\ref{e:dr})), we get
\begin{align*}
(\rho\ostar\om)\ostar\tau&=\frac{1}{4}(\rho\star(\om\star\tau)-(\om\star\tau)\bullet\rho-\rho\star(\tau\bullet\om)+(\tau\bullet\om)\bullet\rho)\\
&=\frac{1}{2}(\rho\ostar(\om\star\tau)-\frac{1}{2}(\rho\star(\tau\bullet\om)-(\tau\bullet\om)\bullet\rho))\\
&=\frac{1}{2}(\rho\ostar(\om\star\tau)-\rho\ostar(\tau\bullet\om))\\
&=\rho\ostar(\frac{1}{2}(\om\star\tau-\tau\bullet\om))\\
&=\rho\ostar(\om\ostar\tau).
\end{align*}
Complete contractivity of the product $\ostar$ follows from that of the individual products. 

The arguments for $\ostar^+$ are identical.
\end{proof}

Inspection of the necessary relations between the dual products used in the above proof leads to the following general result for commuting Banach algebra products.

\begin{prop}\label{p:general} Let $A$ be a Banach space equipped with associative Banach algebra products $\star$ and $\bullet$ satisfying
\begin{enumerate}
\item $a \star (b\bullet c)=b\bullet(a\star c)$,
\item $(a \star b)\bullet c=(a \bullet c) \star b$,
\end{enumerate}
for all $a,b,c\in A$. Then
\begin{align*}a\ostar b&:=\frac{1}{2}(a \star b-b\bullet a), \ \ \ a,b\in A,\\
a\ostar^+b&:=\frac{1}{2}(a\star b+b\bullet a), \ \ \ a,b\in A,
\end{align*}
define associate Banach algebra products on $A$.
\end{prop}

In section \ref{s:last} we study a class of Banach algebras satisfying the conditions of Proposition \ref{p:general}.

We now investigate the Banach algebra structure of $(\TCQ_0,\ostar)$, beginning with some basic observations.

\begin{remark}\label{r:TCQ_0}
\begin{enumerate}
\item More general associativity conditions are valid:
$$(\rho\ostar\om_0)\ostar\tau=\rho\ostar(\om_0\ostar\tau), \ \ \textnormal{and} \ \ (\rho\ostar^+\om_0)\ostar^+\tau=\rho\ostar^+(\om_0\ostar^+\tau),$$
for all $\rho,\tau\in\TCQ$ and $\om_0\in\TCQ_0$. That is, only the middle element needs to lie in the ideal $\TCQ_0$ to guarantee associativity. It follows that $\TCQ$ becomes an operator bimodule over $(\TCQ_0,\ostar)$ and $(\TCQ_0,\ostar^+)$. 

\item If one interchanges the roles of $\G$ and $\wh{\G}$ in the definitions of (\ref{e:jordan}) (respectively, (\ref{e:product})), one obtains the opposite product (respectively, the negative of the opposite product). That is,
$$\rho\:\wh{\ostar^+}\:\tau=\frac{1}{2}(\rho\bullet\tau+\tau\star\rho)=\tau\ostar^+\rho, \ \ \ \rho,\tau\in\TCQ_0.$$
$$\rho\:\wh{\ostar}\:\tau=\frac{1}{2}(\rho\bullet\tau-\tau\star\rho)=-\tau\ostar\rho, \ \ \ \rho,\tau\in\TCQ_0.$$
\item The restriction $\pi_0$ of the canonical quotient map $\pi:\TCQ\to\LOQ$ to $\TCQ_0$ clearly maps into $\LOQ_0=\mathrm{Ker}(1)$, the augmentation ideal of $\LOQ$. The adjoint $\pi_0^*:\LIQ/\C1\to\BLTQ/\C1$ of $\pi_0:\TCQ_0\to\LOQ_0$ is an isometry as 
$$\inf\{\norm{x+\lm1}_{L^\infty(\G)}\mid\lm\in\C\}=\inf\{\norm{x+\lm1}_{\BLTQ}\mid\lm\in\C\}, \ \ \ x\in\LIQ.$$
Thus, $\pi_0$ is a quotient map. Moreover, $\frac{1}{2}\pi_0$ is multiplicative: first, $\pi_0(\tau\bullet\rho)=0$ for all $\tau,\rho\in\TCQ_0$:
$$\la\pi_0(\tau\bullet\rho),x\ra=\la\tau\ten\rho,\wh{\Gamma}(x)\ra=\la\tau\ten\rho,x\ten 1\ra=\tau(x)\rho(1)=0, \ \ \ x\in\LIQ.$$
Therefore,
$$\frac{1}{2}\pi_0(\rho\ostar\tau)=\frac{1}{4}(\pi_0(\rho)\star\pi_0(\tau)))=\frac{1}{2}\pi_0(\rho)\star\frac{1}{2}\pi_0(\tau).$$
Similarly, The restriction $\hat{\pi}_0$ of the canonical quotient map $\hat{\pi}:\TCQ\to\LOQH$ to $\TCQ_0$ is a quotient map onto $\LOQH_0:=\mathrm{Ker}(1)$, and $\frac{1}{2}\hat{\pi}_0$ is multiplicative (up to $-1$) with respect to $\ostar^+$ ($\ostar$).
\item The identity $(\rho\star\tau)\star\om=\rho(\tau\bullet\om+\tau\star\om)$, valid for all $\rho,\tau,\om\in\TCQ_0$, coincides with the \textit{shuffle identity} (16) from \cite{EK}. We thank Roland Speicher for bringing \cite{EK} to our attention in a discussion with the second author on parts of the
present work.
\end{enumerate}
\end{remark}

When $\G$ is finite, $\BLTQ=M_d(\C)$, where $d=\dim\LIQ$. In this case, $\TCQ_0$ is the space of $d\times d$ trace zero matrices, i.e., $\mathfrak{sl}(d)$, the Lie algebra of $SL(d,\C)$. Theorem \ref{t:Jordan} therefore yields new associative products on $\mathfrak{sl}(d)$ for any quantum group of dimension $d$.

\begin{example} Let us explicitly calculate the $\ostar$-product on $\mc{T}(\ell^2(G))_0$ for a finite group $G$. By Examples \ref{ex:1} and \ref{ex:2}, for any $\om=[\om_{s,t}],\tau=[\tau_{s,t}]\in\mc{T}(\ell^2(G))_0$,
$$(\om\star\tau)_{s,t}=\sum_{r\in G}\la\rho(r)^*\om\rho(r)\delta_t,\delta_s\ra\tau_{r,r}=\sum_{r\in G}\tau_{r,r}\om_{sr^{-1},tr^{-1}},$$
$$(\tau\bullet\om)_{s,t}=\tr(\om\lm(ts^{-1}))\tau_{s,t}=\tau_{s,t}\sum_{r\in G}\om_{rs,rt}.$$
Hence, the $(s,t)$-entry of the product is 
$$(\om\ostar\tau)_{s,t}=\frac{1}{2}\sum_{r\in G}(\tau_{r,r}\om_{sr^{-1},tr^{-1}}-\tau_{s,t}\om_{rs,rt}).$$
In words, it is the $\pi(\tau)$-average of the ``right $(s,t)$-diagonal'' of $\omega$ minus the $(s,t)$ entry of $\tau$ times the uniform average of the ``left $(s,t)$-diagonal'' of $\om$. 
It follows from the proof of Theorem \ref{t:bai} below that $E=2(\om_{\delta_e}-\om_{\chi})$ is the identity for the $\ostar$-product, where $\chi(s)=|G|^{-1/2}$, $s\in G$, is the normalized constant function. As a matrix,
$$E=2(\delta_e\delta_e^*-\chi\chi^*)=\frac{-2}{|G|}\begin{bmatrix}1-|G| & 1 & \cdots & 1\\
1 & 1 & \cdots & 1\\
\vdots & \vdots & \cdots & \vdots\\
1 & 1 & \cdots & 1\end{bmatrix}.$$
\end{example}

We now show that under mild conditions on a quantum group $\G$, the algebra $(\TCQ_0,\ostar)$ is abelian if and only if $\mathrm{dim}(\G)=1$ (in which case $\TCQ_0=\{0\}$). In preparation, we require:

\begin{lem}\label{l:abelian} Let $\G$ be a locally compact quantum group such that (at least) one of the following conditions hold: 
\begin{enumerate}
\item $L^1(\G)_0=\la L^1(\G)_0\star L^1(\G)_0\ra$;
\item $\G$ is co-amenable.
\end{enumerate}
Then $L^1(\G)$ is abelian if and only if $L^1(\G)_0$ abelian.
\end{lem}

\begin{proof} One direction trivially holds for any $\G$, so suppose that $L^1(\G)_0$ abelian. 

Assume condition (1) holds, and fix a state $h\in\LOQ$. If $f_0=\sum_{i=1}^n f_{1,i}\star f_{2,i}$ with each pair $f_{1,i},f_{2,i}\in L^1(\G)_0$ then 
$$f_0\star h=\sum_{i=1}^n f_{1,i}\star \underbrace{(f_{2,i}\star h)}_{\in L^1(\G)_0}=\sum_{i=1}^n (f_{2,i}\star h)\star f_{1,i}=\sum_{i=1}^n f_{2,i}\star \underbrace{(h\star f_{1,i})}_{\in\LOQ_0}=\sum_{i=1}^n h\star f_{1,i}\star f_{2,i}=h\star f_0.$$
Since $L^1(\G)_0=\la L^1(\G)_0\star L^1(\G)_0\ra$, the same conclusion is valid for arbitrary $f_0\in\LOQ_0$. 

Now, for any $f,g\in\LOQ$, there exist $f_0,g_0\in L^1(\G)_0$ such that
$$f=f_0+f(1)h \ \ \textnormal{and} \ \ g=g_0+g(1)h.$$
Denoting the commutator of elements in $\LOQ$ by $[\cdot,\cdot]$, one sees that
$$[f,g]=g(1)[f_0,h]+f(1)[h,g_0].$$
By the above reasoning, $[f_0,h]=[h,g_0]=0$, so that $[f,g]=0$ and the claim follows.

Assume condition (2) holds, and let $(h_i)$ be a bounded approximate identity of $\LOQ$ consisting of states \cite[Theorem 2]{HNR}. Then given any $f,g\in\LOQ$, there exist $f_i,g_i\in L^1(\G)_0$ such that
$$f=f_i+f(1)h_i \ \ \textnormal{and} \ \ g=g_i+g(1)h_i.$$
As above, we have
$$[f,g]=g(1)[f_i,h_i]+f(1)[h_i,g_i]=g(1)f(1)([f,h_i]+[h_i,g]).$$
But $[f,h_i],[h_i,g]\to 0$, so we must have $[f,g]=0$, and the proof is complete.
\end{proof}

\begin{remark} It is known that $\LOQ_0$ has a bounded approximate identity if and only if $\G$ is both amenable and co-amenable \cite[Proposition 16]{HNR0}. The existence of a bounded approximate identity for $\LOQ_0$ implies the factorization $\LOQ_0=\LOQ_0\star\LOQ_0$ which in turn implies $\LOQ_0=\la\LOQ_0\star\LOQ_0\ra$.
\end{remark}

\begin{prop}\label{p:abelian} Let $\G$ be a locally compact quantum group such that (at least) one of the following conditions hold: 
\begin{enumerate}
\item $\wh{\G}$ is co-amenable;
\item $\G$ is co-amenable;
\item $L^1(\wh{\G})_0=\la L^1(\wh{\G})_0\star L^1(\wh{\G})_0\ra$;
\item $L^1(\G)_0=\la L^1(\G)_0\star L^1(\G)_0\ra$.
\end{enumerate}
Then if $\mathrm{dim}(\G)\geq 2$, $(\TCQ_0,\ostar)$ is nonabelian.
\end{prop}

\begin{proof} Let $\G$ be such that $\mathrm{dim}(\G)\geq 2$. Assume $(\TCQ_0,\ostar)$ is abelian. Since $\frac{1}{2}\hat{\pi}_0:\TCQ_0\to L^1(\wh{\G})_0$ is a surjective map satisfying
$$\frac{1}{2}\hat{\pi}_0(\rho\ostar\tau)=-\frac{1}{2}\hat{\pi}_0(\rho)\bullet\frac{1}{2}\hat{\pi}_0(\tau), \ \ \ \rho,\tau\in\TCQ_0$$
(by Remark \ref{r:TCQ_0}(3)), it follows that $L^1(\wh{\G})_0$ is abelian. Similarly, since $\frac{1}{2}\hat{\pi}_0:\TCQ_0\to L^1(\G)_0$ is a surjective homomorphism (by Remark \ref{r:TCQ_0}(3)), $\LOQ_0$ is abelian. Therefore, under any of the above four conditions, we can conclude that either $L^1(\wh{\G})$ or $\LOQ$ is abelian by Lemma \ref{l:abelian}. Thus, we may assume without loss of generality that $\G$ is a commutative quantum group, hence $\LIQ=L^\infty(G)$ for a locally compact group $G$ \cite[Theorem 2]{T}.

Since $(\TC_0,\ostar)$ is abelian, for every $\rho,\tau\in\TC_0$ and $T\in\BLT$,
$$\la\rho\ten\tau,\Gam^r(T)-\Sigma\wh{\Gam}^r(T)\ra=2\la\rho\ostar\tau,T\ra=2\la\tau\ostar\rho,T\ra=\la\rho\ten\tau,\Sigma\Gam^r(T)-\wh{\Gam}^r(T)\ra.$$
Fix $\tau\in\TC_0$, $f\in C_0(G)$ and $s\in G$, $s\neq e$. Then with $T=M_f\lambda(s)$, we have
$$\la\rho\ten\tau,\Gam(M_f)(\lambda(s)\ten 1)-(\lm(s)\ten M_f\lm(s))\ra=\la\rho\ten\tau,\Sigma\Gam(M_f)(1\ten \lm(s))-(M_f\lm(s)\ten\lm(s))\ra$$
for all $\rho\in\TC_0$.
Let $\xi$ and $\eta$ be unit vectors in $\LT$, and put $\rho=\om_{\xi}-\om_{\eta}\in\TC_0$. One easily checks that $(\lm(s)\cdot\rho)|_{\LI}=\lm(s)\xi\cdot\overline{\xi}-\lm(s)\eta\cdot\overline{\eta}\in\LO$, so that
\begin{align}\label{e:s}&\la \tau,M_f\star(\lm(s)\xi\cdot\overline{\xi}-\lm(s)\eta\cdot\overline{\eta})-(\la\lm(s)\xi,\xi\ra-\la\lm(s)\eta,\eta\ra)M_f\lm(s)\ra\\
&=\la\tau,((\xi\cdot\overline{\xi}-\eta\cdot\overline{\eta})\star M_f-(\la M_f\lm(s)\xi,\xi\ra-\la M_f\lm(s)\eta,\eta\ra)\lm(s)\ra.
\end{align}
Now, let $g,h\in G$, and take nets $(\xi_i)$ and $(\eta_j)$ of unit vectors in $C_c(G)\subseteq\LT$ such that for any neighbourhoods $U$ and $V$ of $g$ and $h$, respectively, there exist $i_U$ and $j_V$ for which $\mathrm{supp}(\xi_i)\subseteq U$ for $i\geq i_U$ and $\mathrm{supp}(\eta_j)\subseteq V$ for $j\geq j_V$. Since $s\neq e$, there are neighbourhoods $U$ and $V$ of $g$ and $h$, respectively, such that $s^{-1}U\cap U=s^{-1}V\cap V=\emptyset$. Hence, $\lm(s)\xi_i\cdot\overline{\xi_i}, \lm(s)\eta_j\cdot\overline{\eta_j}\to 0$ in $(\LO,\norm{\cdot}_1)$. Taking limits in (\ref{e:s}) It follows that
$$\lim_i\lim_j
\la \tau,M_f\star(\lm(s)\xi_i\cdot\overline{\xi_i}-\lm(s)\eta_j\cdot\overline{\eta_j})-(\la\lm(s)\xi_i,\xi_i\ra-\la\lm(s)\eta_j,\eta_j\ra)M_f\lm(s)\ra=0.$$
On the other hand, $\xi_i\cdot\overline{\xi_i}\to\delta_g$ and $\eta_j\cdot\overline{\eta_j}\to\delta_h$ weak* in $M(G)$, so that $(\xi_i\cdot\overline{\xi_i}-\eta_j\cdot\overline{\eta_j})\star M_f\to (\delta_g-\delta_h)\star M_f$ weak* in $\LI$ (as $f\in C_0(G)$). Also, 
$$\la M_f\lm(s)\xi_i,\xi_i\ra-\la M_f\lm(s)\eta_j,\eta_j\ra=\la M_f,\lm(s)\xi_i\cdot\overline{\xi_i}-\lm(s)\eta_j\cdot\overline{\eta_j}\ra\to 0,$$
so that
$$0=\lim_i\lim_j \la\tau,((\xi_i\cdot\overline{\xi_i}-\eta_j\cdot\overline{\eta_j})\star M_f-(\la M_f\lm(s)\xi_i,\xi_i\ra-\la M_f\lm(s)\eta_j,\eta_j\ra)\lm(s)\ra=\la\tau,((\delta_g-\delta_h)\star M_f)\lm(s)\ra.$$
Since $\tau\in\TC_0$ was arbitrary, it follows that 
$$M_{\rho(g)f}\lm(s)=M_{\rho(h)f}\lm(s)+z1,$$
for some $z\in\C$. But then
$$M_{\rho(g)f}-M_{\rho(h)f}=z\lm(s^{-1})\in\LI\cap VN(G)=\C1,$$
which, as $s\neq e$, forces $z=0$, so that $\rho(g)f=\rho(h)f$. Since $G$ has at least 2 elements, then choosing $g=s$ and $h=e\neq g$, we see that $\rho(s)f=f$. Since $f\in C_0(G)$ was arbitrary, we obtain a contradiction. Hence $G=\{e\}$ is trivial.
\end{proof}

\begin{cor} For any locally compact group $G$, with $|G|\geq 2$, $(\TC_0,\ostar)$ is nonabelian.
\end{cor}

\subsection{Approximate Identities for $\TCQ_0$}


We now study the existence of bounded approximate identities for $(\TCQ_0,\ostar)$, beginning with the following well-known Lemma, whose proof for we include for completeness. 

\begin{lem}\label{l:jhat} Let $\G$ be a locally compact quantum group. If $\G$ is co-amenable, there exists a net $(\eta_i)$ in $\Lphi(\Nphi)_{\norm{\cdot}=1}$, satisfying $J\eta_i=\eta_i$ and 
$$\norm{V(\xi\ten\eta_i)-\xi\ten\eta_i},\norm{W(\eta_i\ten\xi)-\eta_i\ten\xi}\to 0, \ \ \ \xi\in\LTQ.$$
\end{lem}

\begin{proof} It is well-known that co-amenability implies the existence of a (two-sided) co-unit $E\in\LIQ^*$ satisfying 
$$E(x\star f)=\la f, x\ra=E(f\star x), \ \ \ f\in\LOQ, \ x\in\LIQ.$$
The standard argument shows that $(\id\ten E)V=1$ (see, e.g., \cite[Theorem 3.1]{BT}). Approximate $E$ weak* by a net $(f_i)$ of states in $\LOQ$. Since $\LIQ$ is standardly represented on $\LTQ$, each state $f_i=\om_{\eta_i}|_{\LIQ}$ for a unique unit vector $\eta_i\in \mc{P}:=\mc{P}_\vphi=\overline{\{\pi_\vphi(x)J\Lphi(x)\mid x\in\mc{N}_{\vphi}\}}$. By norm density, we may assume $\eta_i=\Lphi(x_i)J\Lphi(x_i)$ for some $x_i\in\Nphi\cap\mc{D}(\sigma^\vphi_{i/2})$, in which case $\eta_i=\Lphi(x_i\sigma^\vphi_{i/2}(x_i)^*)\in\Lphi(\Nphi)$ and $J\eta_i=\eta_i$. Then, for every $\xi\in\LTQ$,
$$\la V(\xi\ten\eta_i),\xi\ten\eta_i\ra\to\la (\id\ten E)(V)\xi,\xi\ra=\norm{\xi}^2,$$
from which it follows that $\norm{V(\xi\ten\eta_i)-\xi\ten\eta_i}\to 0$.

Formula \ref{e:WV} then implies
\begin{align*}\norm{W(\eta_i\ten\xi)-\eta_i\ten\xi}&=\norm{(1\ten U^*)\sigma V\sigma(1\ten U)(\eta_i\ten\xi)-\eta_i\ten\xi}\\
&=\norm{V(U\xi\ten\eta_i)-U\xi\ten\eta_i}\\
&\to0.
\end{align*}


\end{proof}

\begin{cor}\label{c:central} Let $\G$ be a locally compact quantum group. If $\G$ is co-amenable, there exists a net $(\eta_i)$ in $\Lphi(\Nphi)_{\norm{\cdot}=1}$, satisfying
$$\norm{\om_{\eta_i}\star\rho-\rho\bullet\om_{\eta_i}}\rightarrow0, \ \rho\in\TCQ.$$
\end{cor}

\begin{proof} By Lemma \ref{l:jhat}, co-amenability implies the existence of a net $(\eta_i)$ of unit vectors in $\LTQ$ satisfying 
$$\norm{W(\eta_i\ten\xi)-\eta_i\ten\xi}\to0, \ \ \ \xi\in\LTQ.$$
By Proposition \ref{p:products}, for each $i$, we then have
\begin{align*}\om_{\xi}\bullet\om_{\eta_i}&=m_{\bullet}(\om_{\xi}\ten\om_{\eta_i})\\
&=m_{\star}\circ\Ad(W^*)\circ\Sigma(\om_{\xi}\ten\om_{\eta_i})\\
&=m_{\star}(W^*(\eta_i\eta_i^*\ten\xi\xi^*)W).
\end{align*} 
Hence,
$$\om_{\xi}\bullet\om_{\eta_i}-\om_{\eta_i}\star\om_{\xi}=m_{\star}(W^*(\eta_i\eta_i^*\ten\xi\xi^*)W-\eta_i\eta_i^*\ten\xi\xi^*)\to0.$$
By linearity, boundedness and density, the claim follows.
\end{proof}.

\begin{thm}\label{t:bai} $(\TCQ_0,\ostar)$ has a bounded approximate identity if and only if $\G$ and $\wh{\G}$ are co-amenable.
\end{thm}

\begin{proof} Combining Lemma \ref{l:jhat} and (the proof of) Corollary \ref{c:central}, there exists a net $(\eta_i)$ of unit vectors in $\LTQ$ satisfying
\begin{enumerate}[label=(\alph*)]
\item $\norm{V(\xi\ten\eta_i)-\xi\ten\eta_i}\to 0$, $\xi\in\LTQ$; 
\item $\norm{\om_{\eta_i}\star\rho-\rho\bullet\om_{\eta_i}}\to 0$, $\rho\in\TCQ$;
\item $J\eta_i=\eta_i$.
\end{enumerate}
Since $\wh{\G}$ is also co-amenable, the same results imply the existence of a net $(\xi_j)$ of unit vectors in $\LTQ$ satisfying
\begin{enumerate}[label=(\alph*')]
\item $\norm{\wh{V}(\xi\ten\xi_j)-\xi\ten\xi_j}\to 0$, $\xi\in\LTQ$; 
\item $\norm{\om_{\xi_j}\bullet\rho-\rho\star\om_{\xi_j}}\to 0$, $\rho\in\TCQ$;
\item $\wh{J}\xi_j=\xi_j$.
\end{enumerate}
Consider $\rho_{i,j}:=\om_{\xi_j}-\om_{\eta_i}\in\TCQ_0$. Then 
\begin{align*}\tau\ostar\rho_{i,j}&=\frac{1}{2}(\tau\star\rho_{i,j}-\rho_{i,j}\bullet\tau)\\
&=\frac{1}{2}(\tau\star\om_{\xi_j}-\tau\star\om_{\eta_i}-\om_{\xi_j}\bullet\tau+\om_{\eta_i}\bullet\tau).
\end{align*}
By property (b') of the net $(\om_{\xi_j})$, 
$$\lim_j\tau\ostar\rho_{i,j}=\frac{1}{2}(\om_{\eta_i}\bullet\tau-\tau\star\om_{\eta_i}).$$
Next, property $(a)$ of the net $(\om_{\eta_i})$ implies
$$\norm{\om_\xi\star\om_{\eta_i}-\om_\xi}=\norm{(\id\ten\tr)V^*(\xi\xi^*\ten\eta_i\eta_i^*)V-\xi\xi^*}\to0, \ \ \ \xi\in\LTQ.$$
By linearity, density and boundedness, $\rho\star\om_{\eta_i}\to\rho$ for all $\rho\in\TCQ$. Finally, relations (\ref{e:hat}) and (\ref{e:WV}) imply $V=(\wh{J}\ten J)\sigma\wh{V}\sigma(\wh{J}\ten J)$, so that properties (a) and (c) of the net $(\om_{\eta_i})$ imply
$$\norm{\wh{V}(\eta_i\ten\xi)-\eta_i\ten\xi}=\norm{V(\wh{J}\xi\ten\eta_i)-\wh{J}\xi\ten\eta_i}\to 0, \ \ \ \xi\in\LTQ.$$
The above convergence implies that
$$\norm{\om_{\eta_i}\bullet\om_\xi-\om_\xi(1)\om_{\eta_i}}=\norm{(\id\ten\tr)\wh{V}^*(\eta_i\eta_i^*\ten\xi\xi^*)\wh{V}-\om_{\xi}(1)\om_{\eta_i}}\to0, \ \ \ \xi\in\LTQ,$$
which, again by the standard argument shows that $\norm{\om_{\eta_i}\bullet\rho-\rho(1)\om_{\eta_i}}\to0$ for all $\rho\in\TCQ$. In particular, when $\tau\in\TCQ_0$, $\norm{\om_{\eta_i}\bullet\tau}\to0$. Putting these facts together, we see that
$$\lim_i\lim_j\tau\ostar\rho_{i,j}=\frac{1}{2}(\om_{\eta_i}\bullet\tau-\tau\star\om_{\eta_i})=-\frac{1}{2}\tau.$$
Thus, combining the iterated limit into a single net, $(-2\rho_{i,j})$ forms a bounded right approximate identity.

The argument for the left hand side is completely analogous.
\begin{align*}\rho_{i,j}\ostar\tau&=\frac{1}{2}(\rho_{i,j}\star\tau-\tau\bullet\rho_{i,j})\\
&=\frac{1}{2}(\om_{\xi_j}\star\tau-\om_{\eta_i}\star\tau-\tau\bullet\om_{\xi_j}+\tau\bullet\om_{\eta_i}).
\end{align*}
Property (a') of the net $(\om_{\xi_j})$ implies $\tau\bullet\om_{\xi_j}\to\tau$. Property (a') and (c') imply 
$$\norm{V(\xi_j\ten\xi)-\xi_j\ten\xi}=\norm{\wh{V}(J\xi\ten\xi_j)-J\xi\ten\xi_j}\to0, \ \ \ \xi\in\LTQ,$$
which leads to $\norm{\om_{\xi_j}\star\tau}\to0$ for $\tau\in\TCQ_0$. Combining this with property (b) of the net $(\om_{\eta_i})$, we see that
$$\lim_i\lim_j\rho_{i,j}\ostar\tau=\lim_i\frac{1}{2}(\tau\bullet\om_{\eta_i}-\om_{\eta_i}\star\tau-\tau)=-\frac{1}{2}\tau.$$
Thus, $(-2\rho_{i,j})$ forms a bounded approximate identity.

For the converse, suppose that $\TCQ_0$ has a bounded approximate identity $(\rho_i)$. Let $f_i=\pi_0(\rho_i)\in\LOQ_0$. Since $\pi_0$ is a quotient map onto $\LOQ_0$, given $g\in\LOQ_0$, pick $\rho\in\TCQ_0$ such that $\pi_0(\rho)=g$. Note that $\pi(\om\bullet\tau)=0$ for all $\om,\tau\in\TCQ_0$:
$$\la\pi(\om\bullet\tau),x\ra=\la\om\ten\tau,\wh{\Gamma}(x)\ra=\la\om\ten\tau,x\ten 1\ra=\om(x)\tau(1)=0, \ \ \ x\in\LIQ.$$
Hence,
\begin{align*}g\star f_i&=\pi_0(\rho\star \rho_i)=2\pi_0(\frac{1}{2}(\rho\star\rho_i-\rho_i\bullet\rho))=2\pi_0(\rho\ostar\rho_i)\\
&\to2\pi_0(\rho)=2g.
\end{align*}
Similarly, $f_i\star g\to 2g$. Hence, $(\frac{1}{2}f_i)$ forms a bounded approximate identity for $\LOQ_0$. By \cite[Theorem 4.10]{L} (see also \cite[Theorem 4.1]{L}), this entails that $\G$ is both amenable and co-amenable.

By symmetry, $(\frac{1}{2}\hat{\pi}(\rho_i))$ forms a bounded approximate identity for $\LOQH_0$, forcing $\wh{\G}$ to be co-amenable and amenable.

\end{proof} 

\begin{cor} Let $\G$ be a locally compact quantum group. Then $\TCQ_0$ is unital if and only if $\G$ is finite.\end{cor}

\begin{proof} Inspection of the proof of Theorem \ref{t:bai} shows that $\LOQ_0$ has a two-sided identity, call it $e_0$. Then $\LOQ\ni f\mapsto \la 1,f-e_0\star f\ra\in\C$ is a nonzero right $\LOQ$-module map, entailing the compactness of $\G$. Similarly, $\h{\G}$ is also compact, forcing $\G$ to be finite. 
\end{proof}

\section{Completely Bounded Multipliers}

For any locally compact quantum group $\G$, $\TCQ$ is a $(\TCQ_0,\ostar)$-bimodule by Remark \ref{r:TCQ_0}. In this section we show that this bimodule structure captures both $\McbQr$ and $\McbQHr$ via completely isometric anti-representations 
$$\Theta:\McbQr\hookrightarrow \ _{\ostar}\mc{CB}(\TCQ) \ \ \textnormal{and} \ \ \wh{\Theta}:\McbQHr\hookrightarrow\mc{CB}_{\ostar}(\TCQ).$$ 

Recall that an element $\hat{b}'\in\LIQHP$ is a \e{completely bounded right multiplier} of $\LOQ$ if $\rho(f)\hat{b}'\in\rho(\LOQ)$ for all $f\in\LOQ$ and the induced map
\begin{equation*}m_{\hat{b}}^r:\LOQ\ni f\mapsto\rho^{-1}(\rho(f)\hat{b}')\in\LOQ\end{equation*}
is completely bounded on $\LOQ$. We let $\McbQr$ denote the space of completely bounded right multipliers of $\LOQ$, which is a completely contractive Banach algebra with respect to the norm
\begin{equation*}\norm{[\hat{b}'_{ij}]}_{M_n(\McbQr)}=\norm{[m^r_{\hat{b}'_{ij}}]}_{cb}.\end{equation*} 

Given $\hat{b}'\in\McbQr$, the adjoint $\Theta^r(\hat{b}'):=(m_{\hat{b}}^r)^*$ defines a normal completely bounded right $\LOQ$-module map on $\LIQ$, and extends canonically to a normal completely bounded map on $\BLTQ$ \cite[Proposition 4.3]{JNR}. When $\hat{b}'=\rho(f)$, for some $f\in\LOQ$, the map $\Theta^r(\rho(f))$ and its extension are nothing but the convolution actions of $\LOQ$ on $\LIQ$ and $\BLTQ$, respectively, that is,
$$\Theta^r(\rho(f))(x)=(\id\ten f)\Gam(x), \ \ \ \Theta^r(\rho(f))(T)=(\id\ten f)\Gam^r(T), \ \ \ x\in\LIQ, \ T\in\BLTQ.$$
Moreover, the map
\begin{equation}\label{e:Theta}\Theta^r:\McbQr\cong\mc{CB}^{\sigma}_{\star}(\LIQ)\cong\mc{CB}_{\LIQH}^{\sigma,\LIQ}(\BLTQ)\end{equation}
induces a completely isometric isomorphism of completely contractive Banach algebras \cite[Theorem 4.5]{JNR} (we refer to \cite{JNR} for the latter notation).

When $\G$ is co-amenable, $\McbQr=\rho(M(\G))$ \cite[Theorem 4.2]{HNR2}, so we will denote multipliers $\hat{b}'$ simply by $\mu\in M(\G)$. 

We begin with the following observation.

\begin{prop} The maps 
\begin{align*}\Theta&:\McbQr\ni\hat{b}'\mapsto\Theta^r(\hat{b}')_*\in \ _{\ostar}\mc{CB}(\TCQ)\\
\wh{\Theta}&:\McbQHr\ni\hat{\hat{b}'}\mapsto\wh{\Theta}^r(\hat{b}')_*\in \ \mc{CB}_{\ostar}(\TCQ).\end{align*} are completely isometric anti-representations.
\end{prop}

\begin{proof} That $\Theta$ and $\wh{\Theta}$ are completely isometric anti-representations into $\mc{CB}(\TCQ)$ follows from \cite[Theorem 4.5]{JNR}. It remains to show the module properties. Given $\hat{b}'\in\McbQr$, $\rho\in\TCQ_0$ and $\tau\in\TCQ$, we have
$$\Theta(\hat{b}')(\rho_0\ostar\tau)=\frac{1}{2}(\Theta^r(\hat{b}')_*(\rho_0\star\tau)-\Theta^r(\hat{b}')_*(\tau\bullet\rho_0)).$$
Since $\Theta^r(\hat{b}')_*\in \ _{\star}\mc{CB}(\TCQ)$ \cite[Theorem 4.1]{HNR4} and $\Theta^r(\hat{b}')_*\in\mc{CB}_{\bullet}(\TCQ)$ \cite[Proposition 4.2]{C}, we see that
$$\Theta(\hat{b}')(\rho_0\ostar\tau)=\frac{1}{2}(\rho_0\star\Theta^r(\hat{b}')_*(\tau)-\Theta^r(\hat{b}')_*(\tau)\bullet\rho_0)=\rho_0\ostar\Theta(\hat{b}')(\tau).$$
The deduction of the module property of $\wh{\Theta}(b')$, $b'\in\McbQHr$, is completely analogous, appealing to the facts $\wh{\Theta}^r(b')_*\in \ _{\bullet}\mc{CB}(\TCQ)$ and $\wh{\Theta}^r(b')_*\in\mc{CB}_{\star}(\TCQ)$.
\end{proof}

We now characterize the ranges of these maps under amenability assumptions.

\begin{thm}\label{t:multipliers} Let $\G$ and $\wh{\G}$ be co-amenable. Then 
\begin{enumerate}
\item if $\G$ is not compact, $_{\ostar}\mc{CB}(\TCQ)=\Theta(\McbQr)$;
\item if $\G$ is compact with Haar state $\vphi$,  $_{\ostar}\mc{CB}(\TCQ)=\Theta(\McbQr)+\C\lambda(\vphi)\tr(\cdot)$;
\item if $\h{\G}$ is not compact, $\mc{CB}(\TCQ)_{\ostar}=\wh{\Theta}(\McbQHr)$;
\item if $\h{\G}$ is compact with Haar state $\h{\vphi}$, $\mc{CB}(\TCQ)_{\ostar}=\wh{\Theta}(\McbQHr)+\C\h{\lambda}(\h{\vphi})\tr(\cdot)$.
\end{enumerate}
\end{thm}

\begin{proof} Let $\theta\in \ _{\ostar}\mc{CB}(\TCQ)$. Unfolding the product, one easily sees that
\begin{equation}\label{e:module}\theta(\rho_0\star\om)-\rho_0\star\theta(\om)=\theta(\om\bullet\rho_0)-\theta(\om)\bullet\rho_0\end{equation}
for all $\rho_0\in\TCQ_0$ and $\om\in\TCQ$. 

Since $\G$ is co-amenable, Lemma \ref{l:jhat} and Corollary \ref{c:central} imply the existence of a net $(\eta_i)$ of unit vectors in $\LTQ$ satisfying
\begin{enumerate}[label=(\alph*)]
\item $\norm{V(\xi\ten\eta_i)-\xi\ten\eta_i}\to 0$, $\xi\in\LTQ$; 
\item $\norm{\om_{\eta_i}\star\rho-\rho\bullet\om_{\eta_i}}\to 0$, $\rho\in\TCQ$;
\item $J\eta_i=\eta_i$.
\end{enumerate}

Then, as shown in the proof of \ref{t:bai}, we have $\norm{\rho-\rho\star\om_{\eta_i}}\to0$ and $\norm{\om_{\eta_i}\bullet\rho-\rho(1)\om_{\eta_i}}\to0$ for all $\rho\in\TCQ$. In particular, when $\rho_0\in\TCQ_0$, using the module property (\ref{e:module}) we have
\begin{align*}\theta(\rho_0)&=\lim_i\theta(\rho_0\star\om_{\eta_i})\\
&=\lim_i\rho_0\star\theta(\om_{\eta_i})+\theta(\om_{\eta_i}\bullet\rho_0)-\theta(\om_{\eta_i})\bullet\rho_0\\
&=\lim_i\rho_0\star\theta(\om_{\eta_i})-\theta(\om_{\eta_i})\bullet\rho_0.
\end{align*}
Then for any $\tau\in\TCQ$, (not necessarily in $\TCQ_0$), formulae (\ref{e:dr}) imply that
\begin{align*}\tau\star\theta(\rho_0)&=\lim_i\tau\star(\rho_0\star\theta(\om_{\eta_i}))-\tau\star(\theta(\om_{\eta_i})\bullet\rho_0)\\
&=\lim_i\tau\star\rho_0\star\theta(\om_{\eta_i})\\
&=\lim_i(\tau\star\rho_0)\star\theta(\om_{\eta_i})-\theta(\om_{\eta_i})\bullet(\tau\star\rho_0)\\
&=\theta(\tau\star\rho_0)
\end{align*}
(the final equality following by above since $\tau\star\rho_0\in\TCQ_0$). Summarizing, we have 
$$\theta(\tau\star\rho_0)=\tau\star\theta(\rho_0), \ \ \ \tau\in\TCQ, \ \rho_0\in\TCQ_0.$$
In particular, if $i:\TCQ_0\hookrightarrow\TCQ$ is the (left $(\TCQ,\star)$-module) inclusion, the composition $\theta\circ i:\TCQ_0\to\TCQ$ is a left $(\TCQ,\star)$-module map. The adjoint $(\theta\circ i)^*=i^*\circ \theta^*:\BLTQ\to\BLTQ/\C1$ is then a right $(\TCQ,\star)$-module map. Let $m\in\la\LIQ\star\LOQ\ra^*$ be an invariant mean, and fix a state $\hat{f}\in\LOQH$. The map $\Theta^r(m):\BLTQ\to\BLTQ$ defined by
$$\la \Theta^r(m)(T),\rho\ra:=\la m,T\star\rho\ra, \ \ \ T\in\BLTQ, \ \rho\in\TCQ$$
is a right $(\TCQ,\star)$-module conditional expectation onto $\LIQH$ \cite[Theorem 4.2]{CN}. It follows that the state $\tilde{m}:=\hat{f}\circ \Theta^r(m)\in\BLTQ^*$ satisfies
$$\la \tilde{m},T\star\rho\ra=\rho(1)\la \tilde{m},T\ra,\ \ \ T\in\BLTQ, \ \rho\in\TCQ.$$ 
Viewing $\tilde{m}$ as right $(\TCQ,\star)$-module map $\tilde{m}:\BLTQ\to\C1$ (with range inside $\BLTQ$), it follows that $(\id-\tilde{m}):\BLTQ\to\BLTQ$ induces a well-defined right $(\TCQ,\star)$-module map $(\id-\tilde{m}):\BLTQ/\C1\to\BLTQ$. The composition $(\id-\tilde{m})\circ i^*\circ\theta^*:\BLTQ\to\BLTQ$ lies in $\mc{CB}_{\star}(\BLTQ)$ and satisfies 
$$(\id-\tilde{m})\circ i^*\circ\theta^*(T)=\theta^*(T)-\la\tilde{m},\theta^*(T)\ra1, \ \ \ T\in\BLTQ.$$
Since $\G$ is co-amenable, $\mc{CB}_{\star}(\BLTQ)=\Theta^r(\la\LIQ\star\LOQ\ra^*)$ \cite[Proposition 6.5]{HNR2}, so there is some $n\in\la\LIQ\star\LOQ\ra^*$ for which
$$\theta^*(T)=\Theta^r(n)(T)+\la\tilde{m},\theta^*(T)\ra1, \ \ \ T\in\BLTQ.$$
Decomposing $n=\mu+n_{\perp}$ with $\mu\in M(\G)$ and $n_{\perp}\in C_0(\G)^{\perp}$ \cite[Proposition 6.1]{HNR2}, as well as $\tilde{m}\circ\theta^*=\om+m_{\perp}$ with $\om\in\TCQ$ and $m_\perp\in\BLTQ^*\cap\mc{K}(L^2(\G))^{\perp}$ we have
$$\theta^*=(\Theta^r(\mu)+\om(\cdot)1)+(\Theta^r(n_{\perp})+m_{\perp}(\cdot)1).$$
Singularity of $m_{\perp}(\cdot)$ and of $\Theta^r(n_{\perp})$ \cite[Theorem 3.1]{HNR2}, together with normality of $\theta^*, \Theta^r(\mu)$ and $\om(\cdot)$ imply $\theta^*=\Theta^r(\mu)+\om(\cdot)1$. At the predual level, we see that
$$\theta(\rho)=\Theta^r(\mu)_*(\rho)+\la\rho,1\ra\om, \ \ \ \rho\in\TCQ.$$
Since $\theta$ and $\Theta^r(\mu)_*$ are left $\ostar$-module maps, we must have
$$0=\la\rho_0\ostar\rho,1\ra\om=\rho_0\ostar\om\la\rho,1\ra$$
for all $\rho_0\in\TCQ_0$ and $\rho\in\TCQ$, forcing $\om$ to satisfy $\rho_0\ostar\om=0$, equivalently,
$$\rho_0\star\om=\om\bullet\rho_0, \ \ \ \rho_0\in\TCQ_0.$$

The proof will be completed by observing that the $\ostar$-module property of $\theta$ forces $\om$ to either be zero when $\G$ is not compact, or in $\C\lambda(\vphi)$ when $\G$ is compact.

In any case, since $\wh{\G}$ is co-amenable, there exists a net $(\xi_j)$ of unit vectors in $\LTQ$ satisfying
\begin{enumerate}[label=(\alph*')]
\item $\norm{\wh{V}(\xi\ten\xi_j)-\xi\ten\xi_j}\to 0$, $\xi\in\LTQ$; 
\item $\norm{\om_{\xi_j}\bullet\rho-\rho\star\om_{\xi_j}}\to 0$, $\rho\in\TCQ$;
\item $\wh{J}\xi_j=\xi_j$.
\end{enumerate}
As in the proof of Theorem \ref{t:bai} let $\rho_{i,j}:=\om_{\xi_j}-\om_{\eta_i}\in\TCQ_0$. Since $2\rho_{i,j}\ostar\om=0$ for all $i,j$, $\om\bullet\om_{\xi_j}\to\om$ and $\om_{\eta_i}\star\om-\om\bullet\om_{\eta_i}\to 0$, we see that
$$0=\lim_j\lim_i2\rho_{i,j}\ostar\om=\lim_j \om_{\xi_j}\star\om-\om\bullet\om_{\xi_j},$$
implying $\om=\lim_j\om_{\xi_j}\star\om$. But $\om_{\xi_j}\star\om-\om(1)\om_{\xi_j}\to0$, so that $\om=\lim_j\om(1)\om_{\xi_j}$ in trace norm. 

Hence, either $\om\in\TCQ_0$, in which case $\om=0$ since $\TCQ_0$ has a bounded approximate identity, or $\om(1)\neq 0$. In the latter case, $\lim_j\om_{\xi_j}=\om(1)^{-1}\om\in\TCQ$. But then $\om(1)^{-1}\om$ is right identity for $(\TCQ,\bullet)$ and $\wh{\G}$ is necessarily discrete (equivalently, $\G$ is compact) \cite[Proposition 3.7]{KN}. Recalling the construction of $\xi_j$ (from Lemma \ref{l:jhat}), we know that each $\xi_j\in\mc{P}_{\hat{\vphi}}$, the (closed) cone associated to the left Haar weight on $\LIQH$. Since $\LIQH$ is standardly represented on $\LTQ$, the Powers--St\o rmer inequality \cite[Lemma 2.10]{H} holds for all $\xi,\eta\in \mc{P}_{\hat{\vphi}}$:
$$\norm{\xi-\eta}^2\leq\norm{\om_\xi-\om_\eta}\leq\norm{\xi-\eta}\norm{\xi+\eta}.$$
It follows that $(\xi_j)$ is Cauchy in $\LTQ$ so converges to some $\xi\in\mc{P}_{\hat{\vphi}}$. Then $\om(1)^{-1}\om=\om_\xi$ restricts to a right identity in $\LOQH$. But $\om_{\Lambda_{\hat{\vphi}}(\lambda(\vphi))}$ is a two-sided identity in $\LOQH$, so necessarily $\om_\xi|_{\LIQH}=\om_{\Lambda_{\hat{\vphi}}(\lambda(\vphi))}|_{\LIQH}$. By uniqueness of vector representatives in $\mc{P}_{\hat{\vphi}}$, we have $\xi=\Lambda_{\hat{\vphi}}(\lambda(\vphi))$, so that $\om(1)^{-1}\om=\om_{\Lambda_{\hat{\vphi}}(\lambda(\vphi))}$, which, viewed as a trace class operator, is nothing but $\lambda(\vphi)$. 

Regarding the dual statements for $\mc{CB}_{\ostar}(\TCQ)$, observe that
$$\mc{CB}_{\ostar}(\TCQ)=_{-\ostar^{op}}\mc{CB}(\TCQ),$$
so one can simply combine the above reasoning with the fact that $-\ostar^{op}$ is the dual product $\wh{\ostar}$ (by Remark \ref{r:TCQ_0}(2)). This completes the proof.
\end{proof}

\begin{remark} A consequence of the above proof is that for $\G$ and $\wh{\G}$ co-amenable, the left $(\TCQ_0,\ostar)$-module $\TCQ$ is faithful if and only if $\G$ is not compact.
\end{remark}

\begin{cor} Let $\G$ be a co-amenable locally compact quantum group such that $\wh{\G}$ is co-amenable and neither $\G$ nor $\h{\G}$ is compact. Then the maps
\begin{align*}\Theta&:\McbQr\ni\mu\mapsto\Theta^r(\mu)_*\in \ _{\ostar}\mc{CB}(\TCQ)\\
\wh{\Theta}&:\McbQHr\ni\hat{\mu}\mapsto\wh{\Theta}^r(\mu)_*\in \ \mc{CB}_{\ostar}(\TCQ).\end{align*} 
are completely isometric anti-isomorphisms of completely contractive Banach algebras.
\end{cor}

It would be natural to study the completely bounded multipliers of $(\TCQ_0,\ostar)$ in relation to $M_{cb}^r(\LOQ_0)$ and $M_{cb}^r(\LOQH_0)$. We postpone this and related Banach algebra properties of $(\TCQ_0,\ostar)$ to future work.

\section{Nuclear Operators on $L^p(G)$}\label{s:last}
 
Let $G$ be a locally compact group, $p\in (1,\infty)$ and let $\mc{N}(L^p(G))$ denote the nuclear operators on $L^p(G)$. In the second author's thesis \cite{N}, an analogue of the multiplication $\star_a$ on $\mc{T}(L^2(G))=\mc{N}(L^2(G))$ was defined on $\mc{N}(L^p(G))$ via
$$\om\star \tau=\int_G\rho_p(s^{-1})\om\rho_p(s)\pi(\tau)(s)ds, \ \ \ \om,\tau\in\mc{N}(L^p(G)),$$
where $\rho_p$ is the right regular representation on $L^p(G)$ and $\pi:\mc{N}(L^p(G))\to L^1(G)$ denotes pointwise multiplication under the standard identification $\mc{N}(L^p(G))=L^{p}(G)\pten L^{p'}(G)$, where $p'$ is the conjugate variable to $p$ and $\pten$ denotes the Banach space projective tensor product in this section.

There is a fundamental (invertible) isometry $V_p:L^p(G\times G)\to L^p(G\times G)$ given by
$$V_p\xi(s,t)=\Delta(t)^{1/p}\xi(st,t), \ \ \ \xi\in L^p(G\times G),$$
which induces a weak*-weak* continuous co-product $\Gamma^r:\mc{B}(L^p(G))\to\mc{B}(L^p(G\times G))$ via
$$\Gamma(T)=V_p(T\ten 1)V_p^{-1}, \ \ \ T\in\mc{B}(L^p(G)).$$
As observed in \cite[\S2.1]{DS}, $\mc{N}(L^p(G))\ten \mc{N}(L^p(G))$ sits canonically as a dense subspace of $\mc{N}(L^p(G\times G))$ via $(\xi\ten\xi')\ten(\eta\ten\eta')\mapsto (\xi\ten\eta)\ten(\xi'\ten\eta')$. Noting that $V_p^*=V^{-1}_{p'}$, a calculation similar to \cite[(2.1)]{DS} shows that the restriction of the pre-adjoint 
$$\Gamma_*|_{\mc{N}(L^p(G))\ten \mc{N}(L^p(G))}:\mc{N}(L^p(G))\ten \mc{N}(L^p(G))\ni \om\ten\tau\to \om\star\tau\in\mc{N}(L^p(G)).$$

The dual fundamental isometry $\wh{V_p}:L^P(G\times G)\to L^p(G\times G)$ satisfies
$$\wh{V_p}\xi(s,t)=\xi(s,s^{-1}t), \ \ \ \xi\in L^p(G\times G),$$
and coincides with $W_p^{-1}$ in the notation of \cite{DS}. Thus, as shown in \cite{DS}, 
$$\wh{\Gamma}:\mc{B}(L^p(G))\ni T\mapsto \wh{V_p}(T\ten 1)\wh{V_p}^{-1}\in\mc{B}(L^p(G\times G))$$
defines a co-product satisfying $\wh{\Gamma}(\lambda_p(s))=\lambda_p(s)\ten\lambda_p(s)$, $s\in G$, where $\lambda_p$ is the left regular representation on $L^p(G)$. Moreover,
$$\wh{\Gamma}_*|_{\mc{N}(L^p(G))\ten \mc{N}(L^p(G))}:\mc{N}(L^p(G))\ten \mc{N}(L^p(G))\ni \om\ten\tau\to \wh{\Gamma}_*(\om\ten\tau)\in\mc{N}(L^p(G))$$
defines an associative (Banach algebra) product on $\mc{N}(L^p(G))$, which we denote by $\bullet$. 

The dual isometries $V,\wh{V}$ (suppressing the $p$-subscript for notational clarity) satisfy the same commutation relation from the Hilbert space setting: for all $\xi\in L^p(G\times G\times G)$, 
\begin{align*} V_{12}\wh{V}_{13}\xi(r,s,t)&=\Delta(s)^{1/p}\wh{V}_{13}\xi(rs,s,t)\\
&=\Delta(s)^{1/p}\xi(rs,s,s^{-1}r^{-1}t)\\
&=\Delta(s)^{1/p}\wh{V}_{23}\xi(rs,s,r^{-1}t)\\
&=V_{12}\wh{V}_{23}\xi(r,s,r^{-1}t)\\
&=\wh{V}_{13}V_{12}\wh{V}_{23}\xi(r,s,t).
\end{align*}
That is, $V_{12}\wh{V}_{13}=\wh{V}_{13}V_{12}\wh{V}_{23}$. Then, under the canonical trivariate embedding
$$\mc{N}(L^p(G))\ten \mc{N}(L^p(G))\ten  \mc{N}(L^p(G))\subseteq \mc{N}(L^p(G\times G\times G)),$$
we have
\begin{align*}\la (\rho\star\om)\bullet\tau,T\ra&=\la\rho\star\om\ten\tau,\wh{V}(T\ten 1)\wh{V}^{-1}\ra\\
&=\la\rho\ten\om\ten\tau,V_{12}\wh{V}_{13}(T\ten 1\ten 1)\wh{V}_{13}^{-1}V_{12}^{-1}\ra\\
&=\la\rho\ten\om\ten\tau,\wh{V}_{13}V_{12}\wh{V}_{23}(T\ten 1\ten 1)\wh{V}_{23}^{-1}V_{12}^{-1}\wh{V}_{13}^{-1}\ra\\
&=\la\rho\ten\om\ten\tau,\wh{V}_{13}V_{12}(T\ten 1\ten 1)V_{12}^{-1}\wh{V}_{13}^{-1}\ra\\
&=\la\rho\ten\tau\ten\om,\wh{V}_{12}V_{13}(T\ten 1\ten 1)V_{13}^{-1}\wh{V}_{12}^{-1}\ra\\
&=\la(\rho\bullet\tau)\ten\om,V(T\ten 1)V^{-1}\ra\\
&=\la(\rho\bullet\tau)\star\om,T\ra,
\end{align*}
for all $\rho,\om,\tau\in\mc{N}(L^p(G))$, $T\in\mc{B}(L^p(G))$. In other words, condition (2) of Proposition \ref{p:general} is satisfied. Restricting both products to the subspace
$$\mc{N}_{0}(L^p(G)):=\{\om\in\mc{N}(L^p(G))\mid \la\om,1\ra=0\},$$
condition (1) of Proposition \ref{p:general} is also valid. Indeed, the commutation relation
$$\wh{V}_{23}V_{12}\xi(r,s,t)=V_{12}\xi(r,s,s^{-1}t)=\Delta(s)^{1/p}\xi(rs,s,s^{-1}t)=V_{12}\wh{V}_{23}\xi(r,s,t), \ \ \ \xi\in L^p(G\times G\times G),$$
implies that
\begin{align*}\la\rho\star(\om\bullet\tau),T\ra&=\la\rho\ten(\om\bullet\tau),V(T\ten 1)V^{-1}\ra\\
&=\la\rho\ten(\om\bullet\tau),V(T\ten 1)V^{-1}\ra\\
&=\la\rho\ten\om\ten\tau,\wh{V}_{23}V_{12}(T\ten 1\ten 1)V_{12}^{-1}\wh{V}_{23}^{-1}\ra\\
&=\la\rho\ten\om\ten\tau,V_{12}\wh{V}_{23}(T\ten 1\ten 1)\wh{V}_{23}^{-1}V_{12}^{-1}\ra\\
&=\la\rho\ten\om\ten\tau,V_{12}(T\ten 1\ten 1)V_{12}^{-1}\ra\\
&=\la\rho\ten\om\ten\tau,\Gamma(T)\ten 1\ra\\
&=\tau(1)\la\rho\star\om,T\ra\\
&=0
\end{align*}
for all $\rho,\om,\tau\in\mc{N}_{0}(L^p(G))$ and $T\in\mc{B}(L^p(G))$. Proposition \ref{p:general} then shows that the products 
\begin{align*}\rho\ostar \om&=\frac{1}{2}(\rho\star\om -\om\bullet \rho)\\
\rho\ostar^+\om&=\frac{1}{2}(\rho\star \om+\om\bullet \rho)
\end{align*}
define associate Banach algebra products on $\mc{N}_{0}(L^p(G))$.

\section{Outlook}

A number of natural lines of investigation are suggested by our work, which we summarize below.
\begin{enumerate}
\item The dual co-products outlined in the previous section, in combination with \cite{DS} and the developing theory of $L^p$-operator algebras (see, e.g., \cite{GT1,GT2,Ph,PV}) suggest a potential extension of quantum group theory to the setting of $L^p$-spaces. Indeed, as in the case $p=2$, the commutation relation between the dual fundamental isometries $V_{12}\wh{V}_{13}=\wh{V}_{13}V_{12}\wh{V}_{23}$ may be viewed as a type of pentagonal relation. 
\item When is $(\TCQ_0,\ostar)$ (operator) amenable? Even for finite (quantum) groups, can one calculate the amenability constants? See \cite{Choi,Choi2} for recent results on amenability constants of Fourier algebras of finite groups.
\item In line with the previous question, how are homological properties (e.g., flatness, projectivity) of $(\TCQ_0,\ostar)$ related to properties of $\G$? 
\item Determine the centre $\mc{Z}(\TCQ_0,\ostar))$, even for finite quantum groups.
\end{enumerate}

\section*{Acknowledgements}
Jason Crann was partially supported by NSERC Discovery Grant RGPIN-2023-05133 and 
Matthias Neufang by NSERC Discovery Grant RGPIN-2014-06356. This work was finalized while Jason Crann was in residence at Institut Mittag-Leffler in Djursholm, Sweden, during the Operator Algebras and Quantum Information Program 2026. He is grateful for the hospitality and support from the Swedish Research Council under grant no. 2021-06594.


\end{spacing}

\end{document}